\documentclass{amsart}

\usepackage{amsmath,amssymb,amsthm,graphics}
\usepackage[all]{xy}   
\usepackage{subfig}    
\usepackage{hyperref}

\newtheorem{theorem}{Theorem}[section]
\newtheorem{proposition}[theorem]{Proposition}
\newtheorem{conjecture}[theorem]{Conjecture}

\newtheorem{corollary}[theorem]{Corollary}
\theoremstyle{definition}

\newtheorem{remark}[theorem]{Remark}

\newenvironment{italic}{\itshape}{}

\newcommand{\newword}[1]{\textbf{#1}}

\newcommand{\Z}{\mathbb{Z}}

\thanks{The third author gratefully acknowledges the
Funda\c{c}\~ao para a Ci\^encia e a Tecnologia (FCT - PEst- OE/MAT/UI0143/2014) for the support received during the development of this work.}

\begin{document}

\title{Embedding Right-Angled Artin Groups into Brin-Thompson Groups}

\author{James M.~Belk}
\author{Collin Bleak}
\author{Francesco Matucci}

\begin{abstract}
We prove that every finitely-generated right-angled Artin group can be embedded into some Brin-Thompson group~$nV$.  It follows that many other groups can be embedded into some~$nV$ (e.g., any finite extension of any of Haglund and Wise's special groups), and that various decision problems involving subgroups of $nV$ are unsolvable.
\end{abstract}

\maketitle

\section{Introduction}
The \newword{Brin-Thompson groups} $nV$ are an infinite family of groups that act by homeomorphisms on Cantor spaces.  They were first defined by Matt Brin in~\cite{Brin}, and can be viewed as higher-dimensional generalizations of the group~$V$ ($=1V$) defined by Richard J.~Thompson (see \cite{CFP} for an introduction to Thompson's groups).  The groups $nV$ are finitely presented~\cite{HeMa} and simple~\cite{Brin3}, and for $1\leq j<k$ it is known that $jV$ embeds into $kV$ \cite{Brin}, but also that $jV$ and $kV$ are not isomorphic~\cite{BlLa}.

Recall that a \newword{simple graph} is a finite graph with no loops (a loop is an edge from a vertex to itself) and no mulitple edges (any two distinct vertices admit at most one edge between them).  Given a simple graph $\Gamma$, the associated \newword{right-angled Artin group} $A_\Gamma$ has one generator $g_v$ for each vertex $v$ of~$\Gamma$, with one relation of the form $g_vg_w=g_wg_v$ for each pair of vertices $v,w$ that are connected by an edge.  This class of groups has been studied extensively (see~\cite{Char} for an introduction).

Our main result is the following theorem:

\begin{theorem}\label{thm:MainTheorem}For any simple graph\/~$\Gamma$, there exists an $n\geq 1$ so that the right-angled Artin group~$A_\Gamma$ embeds isomorphically into~$nV$.
\end{theorem}

Specifically, we prove that $A_\Gamma$ embeds into $nV$ for $n = |V| + |E^c|$, where $V$ is the set of vertices of~$\Gamma$ and $E^c$ is the set of complementary edges, i.e.~the set of all pairs of vertices that are \textit{not} connected by an edge.

Recall the definition of Bleak and Salazar-D\'iaz in ~\cite{BlSa} that a group $G$ is \newword{demonstrable for a group $K$} of homeomorphisms of a space $X$ if there is an embedding $\hat{G}$ of $G$ into $K$ so that there is an open set $U$ in $X$ so that for all $\hat{g}\neq 1_K\in \hat{G}$ we have $U\cdot \hat{g}\cap U=\emptyset$ (we call the image group $\hat{G}$ a demonstrative subgroup of $K$ with demonstration set $U$).  Bleak and Salazar-D\'iaz also say that such a group $K$ of homeomorphisms of a space $X$ \newword{acts with local realisation} if given any open set $U\subset X$, there is a subgroup $K_U$ of $K$ supported only on $U$ so that $K_U\cong K$.  We now state a theorem of \cite{BlSa}.

{\flushleft {\bf Theorem:} [Bleak--Salazar-D\'iaz, Proposition 3.6]\begin{italic} Let $K$ be a group of homeomorphisms of a space $X$ so that $K$ acts with local realisation, and let $G$ be demonstrable for $K$.  Then $K\wr G$ embeds in $K$. 
\end{italic}}

The embedding $A_\Gamma \to nV$ that we construct is demonstrative for $nV$.  Moreover, $nV$~acts on the Cantor cube with local realization, so we obtain the following.

\begin{theorem}For any simple graph\/~$\Gamma$, there exists an $n\geq 1$ so that the standard restricted wreath product $nV\wr A_\Gamma \cong \bigl(\bigoplus_{A_\Gamma} \!nV\bigr)\rtimes A_\Gamma$ embeds isomorphically into~$nV$.
\end{theorem}

The Krasner-Kaloujnine theorem~\cite{KrKa} states that any extension of a group $G$ by a group $H$ is contained in the (unrestricted) wreath product~$H\wr G$.  Since, for all $n\geq 1$, all finite groups embed in $nV$ in a demonstrative way, Theorem \ref{thm:MainTheorem} and Krasner-Kaloujnine theorem show the following.

\begin{corollary}  For any finite graph\/~$\Gamma$ there is a natural number $n$ so that every finite extension of $A_\Gamma$ embeds into~$nV$.
\end{corollary}

Recall that we say a group $G$ \newword{virtually embeds} in a group $H$ if $G$ admits a finite index subgroup $F$ so that $F$ embeds in $H$.  Note also that any group that virtually embeds into~$nV$ embeds isomorphically into~$nV$.  This follows from the fact that finite groups embed demonstratively into~$V$ (and hence into $nV$ \cite[Lemma~3.3]{BlSa} by taking products across the cantor spaces in the extraneous dimensions), combined with the argument above.  In particular, any group that virtually embeds into a right-angled Artin group embeds isomorphically into some~$nV$.

\begin{corollary}If $G$ is a surface group or a graph braid group, then there exists an $n\geq 1$ so that $G$ embeds into~$nV$.
\end{corollary}
\begin{proof}Droms, Servatius, and Servatious prove that the fundamental group of a surface of genus five embeds into~$A_\Gamma$ for~$\Gamma$ a 5-cycle~\cite{SDS}, and it follows that all hyperbolic surface groups embed into~$10V$. Neunh\"offer, the second and third author observe that all non-hyperbolic surface groups embed into~$V$~\cite{BlMaNe}.

As for graph braid groups, Crisp and Wiest prove that all graph braid groups embed in some right-angled Artin group~\cite{CrWi}.
\end{proof}

Note that the braid groups $B_n$ do not belong to the family of graph braid groups.  We do not know whether braid groups can be embedded into~$nV$.

Haglund and Wise have shown that the fundamental group of any ``special'' cube complex embeds in a right-angled Artin groups~\cite{HaWi1}.  These are known as \newword{special groups}, and any group that has a special subgroup of finite index is \newword{virtually special}.

\begin{corollary}For any virtually special group~$G$, there exists an $n\geq 1$ so that $G$ embeds isomorphically into~$nV$.  This includes:
\begin{enumerate}
\item All finitely generated Coxeter groups~\cite{HaWi2}.\smallskip
\item Many word hyperbolic groups, including all one-relator groups with torsion~\cite{Wise1}.\smallskip
\item All limit groups~\cite{Wise1}.\smallskip
\item Many\/~$3$-manifold groups, including the fundamental groups of all compact\/ \mbox{$3$-manifolds} that admit a Riemannian metric of nonpositive curvature~\cite{PyWi}, as well as all finite-volume hyperbolic $3$-manifolds~\cite{Agol}.
\end{enumerate}
\end{corollary}

Bridson has recently proven several undecidability results for right-angled Artin groups~\cite{Bri}.  These have the following consequences for the Brin-Thompson groups.

\begin{corollary}There exists an $n\geq 1$ with the following properties.  First, the isomorphism problem for finitely presented subgroups of~$nV$ is unsolvable.  Second, there exists a subgroup $H\leq nV$ that has unsolvable subgroup membership problem and unsolvable conjugacy problem.
\end{corollary}

As far as we know, none of the decision problems mentioned in this corollary have been settled for Thompson's group~$V$.  Thus, it is conceivable that the statement of this corollary holds for~$n=1$.

In general, the bound $n = |V| + |E^c|$ for $A_\Gamma$ embedding into $nV$ is far from sharp.  For example, our method proves that a free group of rank~$k$ embeds into $nV$ for $n=k(k+1)/2$, but in fact all such groups embed into~$V$.  However, Bleak and Salazar-D\'{\i}az show that $\mathbb{Z}^2*\mathbb{Z}$ does not embed into~$V$, and hence the only right-angled Artin groups that embed into $V$ are direct products of free groups~\cite{BlSa}.  This leads to the following conjecture.

\begin{conjecture}A right-angled Artin group $A_\Gamma$ embeds into $nV$ if and only if $A_\Gamma$ does not contain $\mathbb{Z}^{n+1}*\mathbb{Z}$.
\end{conjecture}

\begin{remark}
We have several remarks.
\begin{enumerate}
\item Firstly, we note that Corwin and Haymaker in \cite{CoHa} use the main result of Bleak--Salazar-D\'iaz to show that the only obstruction to a right-angled Artin group embedding into $V$ is the existence of a subgroup isomorphic to $\Z^2*\Z$, thus verifying the $n=1$ case of the conjecture above.

\item If the conjecture is true, this would imply that the right-angled Artin group for a $5$-cycle embeds into~$2V$, and hence all surface groups would embed into~$2V$ as well.

\item Hsu and Wise proved that right-angled Artin groups can be embedded into $\mathrm{SL}_n(\mathbb{Z})$ (see~\cite{HsWi}) and
Grigorchuk, Sushanski and Romankov showed that $\mathrm{SL}_n(\mathbb{Z})$ can be realized using synchronous automata. We observe that one can use Theorem~\ref{thm:MainTheorem}
coupled with the embedding Theorem~5.2 in \cite{BeBl} to recover a weaker version of this result, showing that right-angled Artin groups can be realized using asynchronous automata.
\end{enumerate}
\end{remark}

\section{Right-Angled Artin Groups}

Given a finite graph $\Gamma$ with vertex set $V_\Gamma = \{v_1,\ldots,v_n\}$ and edge set~$E$, the corresponding \newword{right-angled Artin group} $A_\Gamma$ is defined by the presentation
\[
A_\Gamma \,=\, \bigl\langle g_1,\ldots,g_n \;\bigl|\; g_ig_j=g_jg_i\text{ for all }\{v_i,v_j\}\in E\bigr\rangle.
\]
For example, if $\Gamma$ has no edges, then $A_\Gamma$ is a free group on $n$ generators.  Similarly, if $\Gamma$ is a complete graph, then $A_\Gamma$ is a free abelian group of rank~$n$.  See \cite{Char} for a general introduction to these groups.

We need a version of the ping-pong lemma for actions of right-angled Artin groups. The following is a slightly modified version of the ping-pong lemma for right-angled Artin groups stated in~\cite{CrFa} (also
see \cite{Kob}).

\begin{theorem}[Ping-Pong Lemma for Right-Angled Artin Groups]
\label{thm:ping-pong}
Let $A_\Gamma$ be a right-angled Artin group with generators $g_1,\ldots,g_n$ acting on a set~$X$.  Suppose that there exist subsets $\{S_i^+\}_{i=1}^n$ and $\{S_i^-\}_{i=1}^n$ of~$X$, with $S_i = S_i^+\cup S_i^-$, satisfying the following conditions:
\begin{enumerate}
\item $g_i(S_i^+) \subseteq S_i^+$ and $g_i^{-1}(S_i^-) \subseteq S_i^-$ for all~$i$.\smallskip
\item If $g_i$ and $g_j$ commute (with $i\ne j$), then $g_i(S_j) = S_j$.\smallskip
\item If $g_i$ and $g_j$ do not commute, then $g_i(S_j) \subseteq S_i^+$ and $g_i^{-1}(S_j) \subseteq S_i^-$.\smallskip
\item There exists a point $x \in X - \bigcup_{i=1}^n S_i$ such that $g_i(x) \in S_i^+$ and $g_i^{-1}(x) \in S_i^-$ for all~$i$.
\end{enumerate}
Then the action of $A_\Gamma$ on $X$ is faithful.\hfill\qedsymbol
\end{theorem}

Indeed, if $U$ is any subset of $X - \bigcup_{i=1}^n S_i$ such that $g_i(U) \subseteq S_i^+$ and $g_i^{-1}(x) \in S_i^-$, then all of the sets $\{g(U) \mid g\in A_\Gamma\}$ are disjoint.  In the case where $X$ is a topological space and $U$ is an open set, this means that the action of $A_\Gamma$ on $X$ is demonstrative in the sense of~\cite{BlSa}.

\section{The Groups $nV$ and $XV$}

Given a finite alphabet $\Sigma$, let $\Sigma^\omega$ denote the space of all strings of symbols from $\Sigma$ under the product topology, and let $\Sigma^*$ denote the set of all finite strings of symbols from $\Sigma$, including the empty string.

A \newword{Cantor cube} is any finite product $X =\Sigma_1^\omega\times \cdots \times \Sigma_n^\omega$, where $\Sigma_1,\ldots,\Sigma_n$ are finite alphabets with at least two symbols each.  Given any tuple $(\alpha_1,\ldots,\alpha_n) \in \Sigma_1^*\times\cdots\times \Sigma_n^*$, the corresponding \newword{subcube} $X(\alpha_1,\ldots,\alpha_n)$ of~$X$ is the set of all points $(x_1,\ldots,x_n)\in X$ such that $x_i$ begin with~$\alpha_i$ for each~$i$.  Note that $X$ is homeomorphic to $X(\alpha_1,\ldots,\alpha_n)$ via the map
\[
(x_1,\ldots,x_n) \;\mapsto\; (\alpha_1\cdot x_1,\ldots,\alpha_n\cdot x_n)
\]
where $\cdot$ denotes concatenation.  More generally, any two subcubes $X(\alpha_1,\ldots,\alpha_n)$ and $X(\beta_1,\ldots,\beta_n)$ of~$X$ have a \newword{canonical homeomorphism} between them given by prefix replacement, i.e.
\[
(\alpha_1\cdot x_1,\ldots,\alpha_n\cdot x_n) \;\mapsto\; (\beta_1\cdot x_1,\ldots,\beta_n\cdot x_n).
\]
A \newword{rearrangement} of a Cantor cube $X$ is any homeomorphism of~$X$ obtained through the following procedure:
\begin{enumerate}
\item Choose a partition $D_1,\ldots,D_k$ of the domain $X$ into finitely many subcubes.\smallskip
\item Choose another partition $R_1,\ldots,R_k$ of the range $X$ into the same number of subcubes.\smallskip
\item Define a homeomorphism $h\colon X\to X$ piecewise by mapping each $D_i$ to $R_i$ via a canonical homeomorphism.
\end{enumerate}
The rearrangements of~$X$ form a group under composition, which we refer to as~$XV$.  In the case where $X = (\{0,1\}^\omega)^n$, the group $XV$ is known as the \newword{Brin-Thompson group $\boldsymbol{nV}$}.

\begin{proposition}If $X = \Sigma_1^\omega\times \cdots \times \Sigma_n^\omega$ is any Cantor cube, then the rearrangement group~$XV$ of~$X$ embeds into the Brin-Thompson group~$nV$.
\end{proposition}
\begin{proof}For each $i$, let $\lambda_{i,1},\lambda_{i,2},\ldots,\lambda_{i,m_i}$ denote the symbols of~$\Sigma_i$. For each $i$ we choose a complete binary prefix code $\alpha_{i,1},\alpha_{i,2},\ldots,\alpha_{i,m_i}$ with $m_i$ different codewords.  That is, we choose a finite rooted binary tree with~$m_i$ leaves, and we let $\alpha_{i,1},\alpha_{i,2},\ldots,\alpha_{i,m_i}$ be the binary addresses of these leaves.  Then the map
\[
\lambda_{i,j_1} \cdot \lambda_{i,j_2} \cdot \lambda_{i,j_3} \cdots \;\mapsto\; \alpha_{i,j_1}\cdot \alpha_{i,j_2} \cdot \alpha_{i,j_3} \cdots
\]
gives a homeomorphism from $\Sigma_i^\omega$ to $\{0,1\}^\omega$.  Taking the Cartesian product gives a homeomorphism $X\to (\{0,1\}^\omega)^n$ which maps each subcube of~$X$ to a subcube of $(\{0,1\}^\omega)^n$, and it is easy to check that conjugating $XV$ by this homeomorphism yields a subgroup of~$nV$.
\end{proof}

\section{Embedding Right-Angled Artin Groups}

The goal of this section is to prove Theorem~\ref{thm:MainTheorem}.  That is, we wish to embed any right-angled Artin group into some~$nV$.

Let $A_\Gamma$ be a right-angled Artin group with generators $g_1,\ldots,g_n$.  For convenience, we assume that none of the generators $g_i$ lie in the center of~$A_\Gamma$.  For in this case $A_\Gamma \cong A'_{\Gamma'}\times \mathbb{Z}$ for some right-angled Artin group $A'_{\Gamma'}$ with fewer generators, and since $sV\times \mathbb{Z}$ embeds in~$sV$, any embedding $A'_{\Gamma'} \to kV$ yields an embedding $A_\Gamma\to kV$.

Let $P$ be the set of all pairs $\{i,j\}$ for which $g_ig_j \ne g_jg_i$, and note that each $i\in\{1,\ldots,n\}$ lies in at least one element of~$P$. Let $X$ be the following Cantor cube:
\[
X \;=\; \prod_{i=1}^n \{0,1\}^\omega \;\times \prod_{\{i,j\}\in P} \{i,j,\emptyset\}^\omega.
\]
Our goal is to prove the following theorem.

\begin{theorem}The group $A_\Gamma$ embeds into $XV$, and hence embeds into $kV$ for $k=n+|P|$.
\end{theorem}

We begin by establishing some notation:
\begin{enumerate}
\item For each point $x\in X$, we will denote its components by $\{x_i\}_{i\in\{1,\ldots,n\}}$ and $\{x_{ij}\}_{\{i,j\}\in P}$.\smallskip
\item Given any $i\in\{1,\ldots,n\}$ and $\alpha\in \{0,1\}^*$, let $C_i(\alpha)$ be the subcube consisting of all $x\in X$ for which $x_i$ begins with~$\alpha$. Let $L_{i,\alpha}\colon X\to C_i(\alpha)$ be the canonical homeomorphism, i.e.~the map that prepends $\alpha$ to~$x_i$.\smallskip
\item For each $i\in\{1,\ldots,n\}$, let $P_i$ be the set of all $j$ for which $\{i,j\}\in P$, and let $S_i$ be the subcube consisting of all $x\in X$ such that $x_{ij}$ begins with $i$ for all $j\in P_i$.  Let $F_i\colon X\to S_i$ be the canonical homeomorphism, i.e.~the map that prepends $i$ to~$x_{ij}$ for each $j\in P_i$.\smallskip
\item Let $S_{ii} = F_i(S_i)=F_i^2(X)$, i.e.~the subcube consisting of all $x\in X$ such that $x_{ij}$ begins with $ii$ for each $j\in P_i$.
\end{enumerate}
Now, for each $i\in\{1,\ldots,n\}$, define a homeomorphism $h_i\colon X\to X$ as follows:
\begin{enumerate}
\item $h_i$ maps $X - S_i$ to $(S_i-S_{ii})\cap C_i(10)$ via $L_{i,10}\circ F_i$.\smallskip
\item $h_i$ is the identity on $S_{ii}$.\smallskip
\item $h_i$ maps $(S_i-S_{ii})\cap C_i(1)$ to $(S_i-S_{ii})\cap C_i(11)$ via $L_{i,1}$.\smallskip
\item $h_i$ maps $(S_i-S_{ii})\cap C_i(01)$ to $X-S_i$ via $F_i^{-1}\circ L_{i,01}^{-1}$.\smallskip
\item $h_i$ maps $(S_i-S_{ii})\cap C_i(00)$ to $(S_i-S_{ii})\cap C_i(0)$ via $L_{i,0}^{-1}$.\smallskip
\end{enumerate}
Note that the five domain pieces form a partition of~$X$, and each is the union of finitely many subcubes.  Similarly, the five range pieces form a partition of~$X$, and each is the union of finitely many subcubes.  Since each of the maps is a restriction of a canonical homeomorphism, it follows that $h_i$ is an element of~$XV$.

\begin{proposition}For each $i,j\in\{1,\ldots,n\}$, if $g_i$ and $g_j$ commute, then so do $h_i$ and $h_j$.
\end{proposition}
\begin{proof}Observe that $h_i(x)$ is completely determined by $x_i$ and $\{x_{ij}\}_{j\in P_i}$, and only changes these coordinates of~$x$.  If $g_i$ and $g_j$ commute, then the relevant sets of coordinates for $h_i$ and $h_j$ do not overlap, and hence $h_i$ and $h_j$ commute.
\end{proof}

Thus we can define a homomorphism $\Phi\colon A_\Gamma \to XV$ by $\Phi(g_i) = h_i$ for each~$i$.

\begin{proposition}The homomorphism $\Phi$ is injective.
\end{proposition}
\begin{proof}For each $i$, let $S_i^+ = S_i \cap C_i(1)$, and let $S_i^{-} = S_i\cap C_i(0)$.  These two sets form a partition of~$S_i$, with
\[
h_i(S_i^+) = S_i\cap C_i(11) \subseteq S_i^+
\qquad\text{and}\qquad
h_i^{-1}(S_i^{-}) = S_i\cap C_i(00) \subseteq S_i^-.
\]
Now suppose we are given two generators $g_i$ and $g_j$.  If $g_i$ and $g_j$ commute, then clearly $h_i(S_j) = S_j$.  If $g_i$ and $g_j$ do not commute, then $S_j \subseteq X-S_i$, and therefore $h_i(S_j)\subseteq S_i^+$ and $h_i^{-1}(S_j)\subseteq S_i^-$.

Finally, let $x$ be an point in $X$ such that $x_{ij}$ starts with $\emptyset$ for all $\{i,j\}\in P$.  Then $x\in X-S_i$ for all~$i$, so $h_i(x) \in S_i^+$ and $h_i^{-1}(x) \in S_i^-$. The homomorphism $\Phi$
is thus injective by Theorem \ref{thm:ping-pong}.
\end{proof}

This proves our main theorem.  Note further that, if $U$ is the open subset of~$X$ consisting of all points $x\in X$ such that $x_{ij}$ starts with $\emptyset$ for all $\{i,j\}\in P$, then $h_i(U) \subseteq S_i^+$ and $h_i^{-1}(U) \subseteq S_i^-$, and therefore all of the sets $\{g(U) \mid g\in A_\Gamma\}$ are disjoint.  Thus the action
of $A_\Gamma$ on $X$ is demonstrative in the sense of~\cite{BlSa}, and it follows that the conjugate action of $A_\Gamma$ on $(\{0,1\}^\omega)^k$ as a subgroup of~$kV$ is demonstrative as well.

\bigskip
\bibliographystyle{plain}

\end{document}